\documentclass[preprint]{siamltex}
\usepackage{amsmath} 
\usepackage{amssymb}  
\usepackage{graphics}
\usepackage{epsfig}
\usepackage[ruled,vlined]{algorithm2e}
\usepackage{subfig}
\usepackage{amsfonts}
\usepackage{mathrsfs}
\usepackage{bbm}
\usepackage{color}
\usepackage{pgfpages}
\usepackage{tikz}
\usepackage{tkz-berge}
\usetikzlibrary{arrows,snakes,backgrounds}

\usepackage[utf8]{inputenc} 

\newtheorem{ex}{Example}

\usepackage{geometry} 
\usepackage{graphicx} 
\usepackage{caption}
\usepackage{booktabs} 
\usepackage{array} 
\usepackage{paralist} 
\usepackage{verbatim} 

\newtheorem{assumption}{Assumption}[section]
\newcommand{\beq}{\begin{equation}}
\newcommand{\eeq}{\end{equation}}

\newcommand{\E}{{\mathcal E}}

\newcommand{\V}{{\mathcal V}}

\newcommand{\R}{\mathbb R}

\newcommand{\one}{\mathbbm{1}}

\def\eps{\varepsilon}

\newcommand{\goes}{\rightarrow}

\newcommand\Lap{\mathrm{Lap}}
\newcommand\sym{\mathrm{Sym}}
\newcommand\Pact {P^+}
\newcommand\Gact{G^+}

\title{Graph partitioning using 
matrix differential equations}

\author{Eleonora Andreotti\footnotemark[1] \and Dominik Edelmann\footnotemark[3]\ ,  \\Nicola Guglielmi\footnotemark[1] \and Christian Lubich\footnotemark[3] }

\begin{document}

\maketitle

\renewcommand{\thefootnote}{\fnsymbol{footnote}}
\footnotetext[1]{Dipartimento di Ingegneria Scienze Informatiche e Matematica, 
Universit\`a degli Studi di L' Aquila,
Via Vetoio - Loc.~Coppito, 
I-67010 L' Aquila and Gran Sasso Science Institute, L'Aquila, Italy. Email: {\tt guglielm@univaq.it}}
\footnotetext[2]{Dipartimento di Ingegneria Scienze Informatiche e Matematica, 
Universit\`a degli Studi di L' Aquila,
Via Vetoio - Loc.~Coppito,
I-67010    L' Aquila,  Italy. Email: {\tt eleonora.andreotti@graduate.univaq.it}}
\footnotetext[3]{Mathematisches Institut,
       Universit\"at T\"ubingen,
       Auf der Morgenstelle 10,
       D--72076 T\"ubingen,
       Germany. Email: {\tt lubich@na.uni-tuebingen.de, dominik.edelmann@na.uni-tuebingen.de}}
\renewcommand{\thefootnote}{\arabic{footnote}}

\begin{abstract}
Given a connected undirected weighted graph, we are concerned with problems related to partitioning the graph. First of all we look for the closest disconnected graph (the minimum cut problem), 
here with respect to the Euclidean norm.
We are interested in the case of constrained minimum cut problems, where constraints
include cardinality or membership requirements, which leads to NP-hard combinatorial optimization problems. 
Furthermore, we are interested in ambiguity issues, that is in the robustness of  clustering algorithms
that are based on Fiedler spectral partitioning. 
The  above-mentioned problems are restated as matrix nearness 
problems for the weight matrix of the graph. A key element in the solution of these matrix nearness problems is the
 use of a constrained gradient system of matrix differential equations.
\end{abstract}



\pagestyle{myheadings}
\thispagestyle{plain}
\markboth{E.~Andreotti, D.~Edelmann, N.~Guglielmi and C.~Lubich}{Graph partitioning using 
differential equations}

\section{Introduction}
In this paper we present a novel approach to partitioning a connected weighted undirected graph. We consider the Frobenius-norm minimum cut problem and allow for constraints such as prescribing the minimum cardinality of connected components or assigning  {\it a priori} selected vertices to a component. We use spectral graph theory as pioneered by Fiedler \cite{fiedler}, see also the monograph by Chung \cite{Ch97} and the introductory articles \cite{Sp07,VL07}. We formulate and use a gradient system of matrix differential equations to drive the smallest nonzero eigenvalue of the graph Laplacian to zero. Once this eigenvalue becomes zero, the graph is disconnected and the corresponding eigenvector indicates the membership of vertices to the connected components. This approach can be extended to other partitioning problems beyond the constrained minimum cut problems considered here.

The approach of this paper takes basic ideas and techniques of recent algorithms for eigenvalue optimization via differential equations, as given for example in \cite{GL11,GKL15,GLM17,GL18},  to another application area. A common feature is a two-level procedure, where on the inner level a gradient flow drives perturbations to the original matrix of a fixed size into a (local) minimum of a functional that depends on eigenvalues and possibly eigenvectors, and in an outer iteration the perturbation size is determined such that the functional becomes zero. As with the previous algorithms cited above, the algorithms presented here cannot guarantee to find the {\it global} minimum of a non-smooth, non-convex optimization problem, or of an NP-hard combinatorial optimization problem.  There are cases where our algorithm could get stuck in a local minimum, and we will present a contrived example where this happens. Even with this caveat, the presented algorithm performs remarkably well in the examples from the literature on which we have tested it.

As opposed to combinatorial algorithms, the algorithm presented here modifies all weights of the graph as it proceeds, and only in the end arrives at the cut and the unchanged remaining weights.  

The proposed algorithm is an iterative algorithm, where in each step the second eigenvalue and the associated eigenvector of the Laplacian of a graph with perturbed weights  are computed. In the cardinality- or membership-constrained cases, additionally a linear system with an extended shifted Laplacian is solved in each step. For a large sparse connected graph (where the number of edges leaving any vertex is moderately bounded), these computations can be done in a complexity that is linear in the number of vertices. In the known (unconstrained) minimum cut algorithms, the computational complexity is at least quadratic \cite{SW97}. It is thus conceivable that for large sparse connected graphs, the proposed iterative algorithm can favorably compete with the classical unconstrained minimum cut algorithms. In constrained cases, it appears that the computational complexity is even more favorable in comparison with the existing heuristic combinatorial algorithms as proposed in \cite{BME04}. However, as of now no detailed comparisons of the relative merits of the conceptually and algorithmically fundamentally different approaches have been made.


In Section~\ref{sec:problem} we formulate the Frobenius-norm minimum cut problem and its cardinality- and membership-constrained variants. This is stated as a matrix nearness problem where it is asked how far, with respect to the Frobenius norm, the weight matrix of the given graph is from that of some disconnected graph which should possibly satisfy additional constraints. We give basic notation and recall Fiedler's theorem on graph connectivity. We also formulate an ambiguity problem where it is asked how far the given weight matrix is from the weight matrix of a graph for which the second and third eigenvalues of the graph Laplacian coalesce and for which therefore graph partitioning based on the Fiedler vector (the eigenvector to the second eigenvalue) becomes ambiguous.

In Section~\ref{sec:two-level method} we describe the two-level approach to the unconstrained Frobenius-norm minimum cut problem. This is the central section of the paper, where the basic approach is developed. 

In Section~\ref{sec:cardinality-membership} we extend the approach to the cardinality- and membership-constrained minimum cut problems, and in Section~\ref{sec:ambiguity} we extend it to the ambiguity problem.

In Section~\ref{sec:alg-aspects} we describe algorithmic aspects such as the discretization of the norm- and inequality-constrained gradient flow,  the choice of initial values, and stopping criteria. In particular, since it is known beforehand that the weights of the cut graph are either zero or those of the original graph, the iteration need not be carried out to full convergence.

Section~\ref{sec:num-examples} shows numerical results of the proposed algorithm for some graphs taken from the literature.

\section{Preparations and problem formulation}
\label{sec:problem}

\subsection{The Frobenius-norm minimum cut problem}
Consider a graph with vertex set $\mathcal{V}=\{1,\dots,n \}$ and edge set $\mathcal{E}\subset \mathcal{V}\times\mathcal{V}$. We  assume that the graph is {\it undirected}: with $(i,j)\in\mathcal{E}$, also $(j,i)\in\mathcal{E}$. With the undirected graph we associate {\it weights} $w_{ij}$ for $(i,j)\in\mathcal{E}$, such that
$$
w_{ij}=w_{ji} \ge 0 \quad\hbox{ for all }\ (i,j)\in\mathcal{E}.
$$
The graph is {\it connected} if for all $i,j\in\mathcal{V}$, there is a path $(i_0,i_1),(i_1,i_2),\dots,(i_{\ell-1},i_\ell)\in\mathcal{E}$ of arbitrary length $\ell$, such that $i=i_0$ and $j=i_\ell$ and $w_{i_{k-1},i_k}>0$ for all $k=1,\dots,\ell$.

The problem considered in this paper is the following: Given a connected weighted undirected graph with weights $w_{ij}$, we aim to find a {\it disconnected} weighted undirected graph with the same edge set $\mathcal{E}$ and modified weights $\widehat w_{ij}$ such that 
\begin{equation}\label{dist-2}
\sum_{(i,j)\in\mathcal{E}} (\widehat w_{ij} - w_{ij})^2 \quad\hbox{ is minimized.}
\end{equation}
The solution to this matrix nearness problem is the same as that of finding a cut  $\mathcal{C}$, i.e., a set of edges that yield a disconnected graph when they are removed from $\E$, where
$$
\text{the cut $\mathcal{C}$ is such that }\ \sum_{(i,j)\in\mathcal{C}} w_{ij}^2 \quad\hbox{ is minimized.}
$$
When the weights are replaced by their square roots, so that $w_{ij}$ instead of $w_{ij}^2$ appears in the above sum, this becomes
the classical minimum cut problem, for which algorithms with complexity $O(|\mathcal{V}|^2 \log |\mathcal{V}| + |\mathcal{V}|\cdot |\mathcal{E}|)$
exist; see Stoer \& Wagner \cite{SW97} and references therein.

\subsection{Constrained minimum cut problems}
The above problem will further be considered with additional constraints. In particular, we consider the following cases:
\begin{itemize}
\item {\it Membership constraint:}\/ It is required that a given set of vertices $\mathcal{V}^+\subset\mathcal{V}$ is in one connected component and another given set of vertices $\mathcal{V}^-\subset\mathcal{V}$ is in the other connected component.
\item {\it Cardinality constraint:}\/ It is required that each of the connected components has a prescribed minimum number $\overline n$ of vertices.
\end{itemize}
It is known that cardinality constraints make the problem NP-hard \cite{BEHM07,BME04}.

\subsection{Graph Laplacian and algebraic connectivity}
Setting $w_{ij}=0$ for $(i,j)\notin \mathcal{E}$, we have the symmetric weight matrix
$$
W=(w_{ij}) \in \R^{n\times n}.
$$
The degrees  $d_i = \sum_{j=1}^n w_{ij} $ are collected in the diagonal matrix 
$$
D = \diag(d_i)= \diag(W  \one), \qquad\hbox{where $\one:=(1,\ldots,1)^T \in \R^n$.}
$$
The {\it Laplacian matrix} $L = \mathrm{Lap}(W)$ is defined by
$$
L = D-W, \quad\mbox{ i.e., }\quad   \mathrm{Lap}(W)=\diag(W  \one) - W .
$$
We note that by the Gershgorin circle theorem, all eigenvalues of $L$ are nonnegative, and 
$L\one=0$, so that $\lambda_1=0$ is the smallest  eigenvalue of $L$. Remarkably, the connectivity of the graph is characterized by the second-smallest eigenvalue of $L$.

\begin{theorem}[M. Fiedler \cite{fiedler}] \label{thm:fiedler}
 Let $W \in \R^{n \times n}$ be the weight matrix of an undirected graph and $L$ the corresponding Laplacian matrix. Let $0 = \lambda_1 \le \lambda_2 \le \ldots \le \lambda_n$ be the eigenvalues of $L$. Then, the graph is disconnected if and only if $\lambda_2 = 0$. Moreover, if $0=\lambda_2<\lambda_3$, then the entries of the corresponding eigenvector orthogonal to $\one$ assume only two different values, of different sign, which mark the membership to the two connected components.
\end{theorem}

Because of this result, the second smallest eigenvalue $\lambda_2$ of $L$ is called {\it algebraic connectivity} of $W$. If $\lambda_2$ is a simple eigenvalue, then the corresponding eigenvector is known as the {\it Fiedler vector}. 

\subsection{An ambiguity problem in graph partitioning}
Based on Theorem~\ref{thm:fiedler}, a common and computationally inexpensive strategy for partitioning a graph is to compute the Fiedler vector and to partition the graph according to the values of its entries. This becomes unreliable when a small perturbation of the weights yields a coalescence of the eigenvalues $\lambda_2$ and $\lambda_3$. It is then interesting to know the distance of the given weight matrix from the set of weight matrices with $\lambda_2=\lambda_3$.

\section{Two-level method for the Frobenius-norm minimum cut problem}
\label{sec:two-level method}

\subsection{Two-level formulation}
Our approach can be summarized as follows:
\begin{enumerate}
\item Given $\eps>0$, we look for a symmetric matrix $E=(e_{ij})\in\R^{n\times n}$ with the same sparsity pattern as $W$ (i.e., $e_{ij}= 0$ if $w_{ij}=0$), of unit Frobenius norm,  with $W+\eps E\ge 0$ (with componentwise inequality)  such that the second smallest eigenvalue of $\mathrm{Lap}(W+\eps E)$ is minimal. The obtained minimizer is denoted by $E(\eps)$.
\item We look for the smallest value of $\eps$ such that the second smallest eigenvalue of 
$\mathrm{Lap}(W+\eps E(\eps))$ equals $0$.
\end{enumerate}

In order to compute $E(\eps)$ for a given $\eps>0$, we make use of a constrained gradient system for the functional
\begin{equation}\label{F-eps}
F_\eps(E) = \lambda_2\bigl( \mathrm{Lap}(W+\eps E) \bigr),
\end{equation}
under the constraints of unit Frobenius norm and $W+\eps E\ge 0$ and the symmetry and the sparsity pattern of~$E$.

In the outer iteration we compute the optimal $\eps$, denoted $\eps^\star$, by a combined Newton-bisection method.

The algorithm computes a partition of the graph as provided by the Fiedler vector corresponding to the weight matrix $W+\eps^\star E(\eps^\star)$. This is not guaranteed to yield a global optimum for the Frobenius-norm minimum cut problem, since the gradient flow might converge only to a local minimum. In any case, it provides an upper bound for the distance problem \eqref{dist-2}.

\subsection{Constrained gradient flow for the functional $F_\eps$}

\subsubsection{Eigenvalue derivatives}
We will use the following standard perturbation result for
eigenvalues; see, e.g., \cite[Section II.1.1]{Kat95}. Here and in the following, we denote $\dot{\phantom{a}}= d/dt$.

\begin{lemma} \label{lem:eigderiv} Consider the differentiable symmetric 
$n\times n$ matrix valued function $C(t)$ for $t$ in a neighborhood of $0$. 
Let $\lambda(t)$ be an eigenvalue of $C(t)$
converging to a simple eigenvalue $\lambda_0$ of $C_0=C(0)$ as $t\goes 0$.  
Let $x_0$ be the associated eigenvector, with $\| x_0 \|_2=1$.
Then $\lambda(t)$ is differentiable near $t=0$ with
$$
   \dot\lambda(0) = x_0^T \dot{C}(0) x_0.
$$
\end{lemma}

\subsubsection{Gradient of $F_\eps$}
\label{subsec:gradient}

We denote by $\|\cdot\|=\|\cdot\|_F$ the Frobenius norm on $\R^{n\times n}$ and by $\langle X,Y \rangle = \mathrm{trace}(X^T Y)$ the corresponding inner product.

We return to the situation of the previous section. For a set of edges $\mathcal{E}$, we define $P_\mathcal{E}$ as the orthogonal projection from $\R^{n\times n}$ onto the sparsity pattern determined by~$\mathcal{E}$: for $A=(a_{ij})$,
\[ 
P_\mathcal{E}(A)\big|_{ij} := \begin{cases}
	a_{ij}\,, &\text{if } (i,j)\in\mathcal{E} \,,\\
	0\,, &\text{otherwise.}\end{cases}
\]

For a fixed given weight matrix $W$ and for $\eps>0$, we call a matrix $E=(e_{ij})\in\R^{n\times n}$ {\it $\eps$-feasible} if the following conditions are satisfied:
\begin{itemize}
\item[(i)] $E$ is of unit Frobenius norm.
\item[(ii)] $E$ is symmetric.
\item[(iii)] $E=P_\E(E)$.
\item[(iv)] $W+\eps E \ge 0$.
\end{itemize}
Consider now a regular path $E(t)$ of $\eps$-feasible matrices, and denote the corresponding Laplacian matrix by $L(t)= \Lap(W+\eps E(t))$ and by $\lambda_2(t)$ the second smallest eigenvalue of $L(t)$.
Lemma \ref{lem:eigderiv} applied to the Laplacian matrix $L(t)$ yields (omitting the argument $t$)
\begin{equation}\label{milestone1}
\dot\lambda_2 = x^T \dot{L} x = \langle x x^T, \dot{L} \rangle,
\end{equation}
where  $x(t)$ is a corresponding eigenvector of unit Euclidean norm. 
Next we rearrange \eqref{milestone1} to an equation $\dot\lambda_2 = \eps \langle G_\eps(E),\dot{E} \rangle $ with an appropriate matrix-valued function~$G_\eps$, which is the gradient of $F_\eps$ in the space of symmetric matrices with sparsity pattern~$\E$.
 In the following, $\sym(A)=\tfrac12(A+A^T)$ denotes the symmetric part of a quadratic matrix~$A$, and we write $x^2=(x_i^2)\in\R^n$ for the  vector of squares of the entries of $x=(x_i)\in\R^n$.
\begin{lemma}\label{lem:a-dot} In the above situation we have
\begin{align}\label{G-eps}
	\dot\lambda_2 = \eps \langle  G_\eps(E) , \dot{E} \rangle, \quad\hbox{ where }\quad
	G_\eps(E)=P_{\mathcal{E}}(\sym(x^2 \one ^T)-x x^T)
\end{align}
is symmetric and has the sparsity pattern determined by the set of edges $\mathcal{E}$. 
\end{lemma}

\begin{proof} We note that
\begin{align}\label{milestone2}
\dot L =	\Lap\left( \frac{d}{dt} \bigl( W+\eps E(t) \bigr)\right) = \eps\, \Lap( \dot{E})  = \eps ( \diag(\dot{E}\,\one ) - \dot{E} ) \,.
\end{align}
Combining \eqref{milestone1} and \eqref{milestone2}, we obtain
\begin{align}\label{milestone3}
	\dot\lambda_2 = \eps \left( \langle x x^T, \diag(\dot{E}\,\one )\rangle - \langle x x^T, \dot{E} \rangle \right) .
\end{align}
The second term is already in the desired form. We obtain for the first term
\begin{align} \label{xx}
	\langle x x^T , \diag(\dot{E}\,\one ) \rangle &=  \sum_{i=1}^n x_i^2 (\dot{E}\,\one )_i = \sum_{i=1}^n \sum_{j=1}^n x_i^2 \one _j \dot{e}_{ij} =\langle x^2 \one ^T, \dot{E} \rangle \,.
\end{align}
This yields
\begin{align*}
	\dot\lambda_2 =  \eps \langle  x^2 \one ^T - x x^T, \dot{E} \rangle \,.
\end{align*}
Since $\dot{E}$ and $x x^T$ are symmetric, this can be rewritten as
\begin{align*}\label{milestone6}
	\dot\lambda_2 = \eps\langle  \sym(x^2 \one ^T) - x x^T , \dot{E} \rangle \,.
\end{align*}
We then have
$$
	\dot\lambda_2 = \eps \langle  G_\eps(E) , \dot{E} \rangle \quad\hbox{ with }\quad
	G_\eps(E)=P_{\mathcal{E}}(\sym(x^2 \one ^T)-x x^T).
$$
This is in the desired form: $G_\eps(E)$ is symmetric and has the  sparsity pattern $\mathcal{E}$. 
\end{proof}

\subsubsection{Admissible directions}
Since $E(t)$ is of unit Frobenius norm by condition (i), we have
\[ 
0 = \frac12\, \frac{d}{dt} \lVert E(t) \rVert^2 =  \langle E(t), \dot E(t) \rangle .
\]
Condition (iv) requires that $\dot{e}_{ij}  \ge 0$ for all $(i,j) \in \mathcal{E}_0$, where $\mathcal{E}_0=\E_0(\eps E)$ is the set of cut edges defined by
\[ 
\mathcal{E}_0 := \{(i,j) \in \mathcal{E}:\, w_{ij} + \varepsilon e_{ij} = 0 \}.
\]
Conditions (ii) and (iii) are satisfied if the same holds for $\dot{E}$. These four conditions are in fact also sufficient for a matrix to be the time derivative of a path of $\eps$-feasible matrices.
Hence, for every $\eps$-feasible matrix $E$, a matrix $Z=(z_{ij})\in\R^{n\times n}$ is the derivative at $t=0$ of some path of $\eps$-feasible matrices starting at $E$ if and only if the following four conditions are satisfied:
\begin{itemize}
\item[(i')] $\langle E,Z\rangle=0$.
\item[(ii')] $Z$ is symmetric.
\item[(iii')] $Z=P_\E(Z)$.
\item[(iv')] $P_{\E_0}(Z)\ge 0$.
\end{itemize}
Condition (iv') says that $z_{ij} \ge 0 \text{ for all }(i,j) \in \mathcal{E}_0$.

\subsubsection{Admissible direction of steepest descent}
To determine the admissible direction $\dot E$ of steepest descent from $E$, we  therefore consider the following optimization problem for $G=G_\eps(E)$:
\begin{align}\label{opt_problem_1}
\min_{Z} \langle G,Z \rangle \quad\text{subject to (i')--(iv') and $\langle Z,Z \rangle = 1$.} 
\end{align}
The additional constraint  $\lVert Z \rVert = 1$ just normalizes the descent direction. 
Problem \eqref{opt_problem_1} has a quadratic constraint. We now formulate a quadratic optimization problem with linear constraints, which is equivalent in the sense that it yields the same descent direction, provided that a strict descent direction exists, i.e., satisfying $\langle G,Z \rangle<0$ and the constraints (i')--(iv'). This is based on the fact that when  $\langle G, Z \rangle < 0$, there exists a scaling factor $\alpha>0$ such that $\langle G, \alpha Z \rangle = -1$. Consider the following problem:
\begin{align}\label{opt_problem_2}
	\min_{Z} \langle Z,Z \rangle \quad \text{subject to (i')-(iv') and $\langle G, Z \rangle = -1$.} 
\end{align}
Both optimization problems yield the same Karush--Kuhn--Tucker (KKT) conditions (apart from the normalization). Since the objective function $\langle Z,Z \rangle$ of problem \eqref{opt_problem_2} is convex and all constraints are linear, the KKT conditions are not only necessary but also sufficient conditions (\cite[Theorem 9.4.1]{F13}),  that is, a KKT point is already a solution of the optimization problem.

The solution of \eqref{opt_problem_2} satisfies the KKT conditions 
\begin{subequations}\label{KKT}
\begin{align}
	&
	   Z = - G - \kappa E + \sum_{(i,j) \in \mathcal{E}_0} \mu_{ij} e_i e_j^T \,, \\
	&\mu_{ij} z_{ij} = 0 \text{ for all }(i,j) \in \mathcal{E}_0 \,, \\
	&\mu_{ij} \ge 0 \text{ for all }(i,j) \in \mathcal{E}_0 \,.
\end{align}
\end{subequations}
In addition, there are conditions (i')--(iv'). It can be shown (see \cite[Lemma 3.2.1]{edelmann-master}) that the symmetry and sparsity conditions (ii') and (iii') need not be imposed, but are consequences of conditions (ii) and (iii) on $E$ and the corresponding properties of $G=G_\eps(E)$.

\subsubsection{Constrained gradient flow}
The gradient flow of $F_\eps$ under the constraints (i)--(iv) is the system of differential equations
\begin{equation}\label{cgf}
\dot E(t) = Z(t),
\end{equation}
where $Z(t)$ solves the KKT system \eqref{KKT} with $G=G_\eps(E(t))$ under the constraints (i')--(iv') with the set of edges $\E_0(t)=\E_0(\eps E(t))$. 

\begin{lemma} On an interval where $\E_0(t)$ does not change, the gradient system becomes, with $\Pact =P_{\E \setminus \E_0}$ and omitting the ubiquitous argument~$t$,
\begin{equation}\label{Pact-ode}
\dot E = - \Pact  G_\eps(E) - \kappa \Pact  E \quad\hbox{ with }\quad
\kappa = \frac{\langle - G_\eps(E), \Pact  E \rangle }{\| \Pact  E\|^2}.
\end{equation}
\end{lemma}

\begin{proof} The positive Lagrange multipliers $\mu_{ij}>0$ just have the role to ensure that $\dot e_{ij}=0$.  With $G=G_\eps(E)$, the gradient system therefore reads
\begin{equation}\label{E-ode}
\dot E = \Pact (-G - \kappa E),
\end{equation}
where $\kappa$ is determined from the constraint $\langle E,\dot E \rangle =0$. We then have
$$
0=\langle E,\dot E \rangle = \langle E,\Pact (-G - \kappa E) \rangle=
- \langle \Pact  E,G \rangle -  \kappa\langle \Pact  E,\Pact  E \rangle,
$$
and the result follows.
\end{proof}

In a numerical solution of the gradient system, we thus have to monitor the sets of edges
where $w_{ij}+\eps e_{ij}=0$ and  among them further those edges where the sign of $-g_{ij}-\kappa e_{ij}$ changes. When the active set is changed, then also $\kappa$ changes in a discontinuous way. 
Let $\kappa_-$ and $\kappa_+$ be the values of $\kappa$ before and after the event of discontinuity, respectively. Then one has generically $g_{ij}+\kappa_- e_{ij}>0$ after the event for the critical edge $(i,j)$, but the sign of 
$g_{ij}+\kappa_+ e_{ij}$ may be positive or negative. In the first case, $(i,j)$ leaves $\E_0$. In the latter case, only a generalized solution in the Filippov sense exists, which keeps $(i,j)\in\E_0$, i.e., $w_{ij}+\eps e_{ij}=0$. This is enforced until $g_{ij}+\kappa_+ e_{ij}$ changes sign. From a practical perspective, it appears reasonable just to keep $(i,j)\in\E_0$ for all future time once it has entered $\E_0$, which means that a cut of an edge is made irreversible.

\subsubsection{Monotonicity and stationary points}

The following monotonicity result follows directly from the construction of the gradient system.

\begin{theorem}
Let $E(t)$ of unit Frobenius norm satisfy the differential equation \eqref{E-ode} with $G_\eps(E)$ of \eqref{G-eps}. Then, the second smallest eigenvalue $\lambda_2(t)$ of the Laplacian matrix $\Lap(W+\eps E(t))$ decreases monotonically with $t$:
$
\dot \lambda_2(t) \le 0.
$
\end{theorem}

Equilibrium points of \eqref{E-ode} are characterized as follows.

\begin{theorem} \label{thm:stat}
The following statements are equivalent along solutions of \eqref{E-ode}:
\begin{enumerate}
\item $\dot \lambda_2 =0$.
\item $\dot E =0$.
\item $\Pact E$ is a real multiple of $\Pact G_\eps(E)$.
\end{enumerate}
\end{theorem}

\begin{proof} Using Lemma~\ref{lem:a-dot} and \eqref{E-ode} we obtain, with $G=G_\eps(E)$,
$$
 \frac1\eps\, \dot\lambda_2 = \langle G, \dot E \rangle 
 = \langle G, - \Pact G - \kappa \Pact E \rangle
 = - \| \Pact G \|^2 + 
 \frac{\langle \Pact G, \Pact E \rangle^2}{\|\Pact E\|^2}.
$$
With the strong form of the Cauchy--Schwarz inequality, the result follows.
\end{proof}

\subsection{Newton-bisection outer iteration}
Let $E(\eps)$ denote the minimizer of the functional $F_\eps$.
In general we expect that for a given perturbation size $\eps<\eps^\star$,  the eigenvalue ${\lambda_2 (W+\eps E(\eps))> 0}$
is simple. 
%
If so, then $f(\eps)=F_\eps(E(\eps))$ is a piecewise smooth function of $\eps$ and we 
can exploit its regularity to obtain a fast iterative method to converge to
$\eps^\star$ from the left.
Otherwise we can use a bisection technique  to approach $\eps^\star$.

The following result provides an inexpensive formula for the computation
of the derivative of $f(\eps)=F_\eps(E(\eps))$, which will be useful in
the construction of the outer iteration of the method. 

\begin{assumption}
We assume that the second smallest eigenvalue  of $\Lap(W + \eps E(\eps))$  
is \emph{simple}. Moreover,  $E(\eps)$ is assumed to be  a smooth function of $\eps$ in some interval, and the set of zero-weight edges $\E_0$ related to $E(\eps)$ is independent of $\eps$ in the interval. 
\label{assumpt}
\end{assumption}

We denote again $\Pact =P_{\E\setminus\E_0}$. We then have the following result.

\begin{lemma}
\label{lem:der}
Under Assumption~{\rm \ref{assumpt}}, the function $f(\eps)=F_\eps(E(\eps))$ is differentiable and its derivative equals (with ${\phantom{a}'}= d/d\eps$)
\begin{equation}
f'(\eps) =  - \| \Pact  G_\eps(E(\eps)) \| \; \| \Pact   E(\eps) \| -
\frac1{\eps^2}\, \frac{\| \Pact  G_\eps(E(\eps)) \| }{ \| \Pact   E(\eps) \|}\, 
\| P_{\E_0}W\|^2.
\label{eq:derFdeps}
\end{equation}
\end{lemma}
\medskip
\begin{proof} Under Assumption~{\rm \ref{assumpt}}, the projection $\Pact $ remains constant near $\eps$. This means that we are effectively working on a reduced set $\widehat\E=\E\setminus\E_0$ of edges. We set $\Gact(\eps)=\Pact  G_\eps(E(\eps))$ and decompose 
$$
\eps E(\eps)=\eps \Pact  E(\eps)+R, \quad\hbox{where}\quad 
R=(I-\Pact )\eps E(\eps)=-(I-\Pact )W=-P_{\E_0}W,
$$ 
since $w_{ij}+\eps e_{ij}(\eps)=0$ for all $(i,j)\in \E_0$. In particular, $R$ is independent of~$\eps$.
Differentiating $f(\eps)=F_\eps\bigl( E(\eps) \bigr)$ with respect to $\eps$ we obtain
\begin{eqnarray}
\hskip -9mm
f'(\eps) =\bigl\langle G_\eps(E(\eps)), \Pact  E(\eps)+\eps \Pact  E'(\eps) \bigr\rangle
=\bigl\langle \Gact(\eps), \Pact  E(\eps)+\eps \Pact  E'(\eps) \bigr\rangle.
\label{eq:dFdeps}
\end{eqnarray}
The  conservation of $\|  E(\eps) \|=1 $ and of $\| R \|$ for all  $\eps$  implies
$$
\langle \Pact  E(\eps), \Pact  E'(\eps)\rangle = \frac{d}{d\eps}\,\tfrac12\,\| \Pact  E(\eps) \|^2 =
\frac{d}{d\eps}\,\tfrac12\, (1-\eps^{-2}\|R\|^2) =\eps^{-3}\|R\|^2.
$$
Now we use the property of minimizers as stated by Theorem~\ref{thm:stat},
$$
\frac{\Gact(\eps)}{\|\Gact(\eps)\|} = \pm \frac{\Pact  E(\eps)}{\|\Pact  E(\eps)\|},
$$
which gives us
$$
\bigl\langle \Gact(\eps), \Pact  E(\eps) \bigr\rangle = \pm \|\Gact(\eps)\| \; \| \Pact  E(\eps)\|
$$
and
\begin{align*}
\bigl\langle \Gact(\eps), \eps \Pact  E'(\eps) \bigr\rangle &=
\pm \eps \frac{\| \Gact (\eps) \| }{ \| \Pact   E(\eps) \|}\,
\langle \Pact  E(\eps), \Pact  E'(\eps)\rangle 
\\
&= \pm \eps^{-2} \frac{\| \Gact (\eps) \| }{ \| \Pact   E(\eps) \|}\, \|R\|^2.
\end{align*}
From
\eqref{eq:dFdeps} we thus obtain the stated formula, 
since $f'(\eps)\le 0$.
\end{proof}


For $\eps = \eps_k < \eps^\star$, we make use of the standard Newton iteration
\begin{equation}
\eps_{k+1} = \eps_k - \frac{f(\eps_k)}{f'(\eps_k)},
\label{eq:Newton}
\end{equation}
In a practical algorithm it is useful to couple the Newton iteration \eqref{eq:Newton} with
a bisection technique. To do this we adopt a tolerance tol which allows us to distinguish
whether $\eps < \eps^\star$, in which case we may use the derivative formula and
perform the Newton step, or $\eps > \eps^\star$, so that we have to make use of bisection.
The method is formulated in Algorithm \ref{alg_dist}.

\IncMargin{1em}
\begin{algorithm}
\DontPrintSemicolon
\KwData{Matrix $W$ is given, $k_{\max}$ (max number of iterations), 
tol (tolerance)\; $\eps_0$, $\eps_{\rm lb}$ and $\eps_{\rm ub}$ (starting values for the lower and upper
bounds for $\eps^\star$)}
\KwResult{$\eps^\star$ (upper bound for the distance), $E(\eps^\star)$}
\Begin{
\nl Compute $E(\eps_0)$ by the inner iteration\;
\nl Set $k=0$\;
\While{$k \le k_{\max}$}{
\eIf{$f(\eps_k) < {\rm tol}$}
{Set $ \eps_{\rm ub} = \min(\eps_{\rm ub},\eps_k)$\;
Set $\eps_{k+1} = (\eps_{\rm lb} + \eps_{\rm ub})/2$ \ (bisection step)}
{Set $ \eps_{\rm lb} = \max(\eps_{\rm lb},\eps_k)$\;
\nl Compute $f(\eps_k)$ and $f'({\eps_k})$\;
\nl Compute $\eps_{k+1} = \displaystyle{\eps_k - \frac{f(\eps_k)}{f'(\eps_k)}}$ \ (Newton step)}
\If{$\eps_{k+1} \not\in (\eps_{\rm lb},\eps_{\rm ub})$}{Set $\eps_{k+1} = (\eps_{\rm lb} + \eps_{\rm ub})/2$}
\eIf{$k=k_{\max}$ {\rm \bf or} $\eps_{\rm ub}-\eps_{\rm lb} < {\rm tol}$}
{Return $\eps_{k+1}$ {\rm \bf and} the interval $[\eps_{\rm lb},\eps_{\rm ub}]$\; {\rm \bf Stop}}
{Set $k=k+1$}
\nl Compute $E(\eps_k)$ by the inner iteration\;
}
}
\nl Return $\eps^\star = \eps_k$\;
\caption{Newton-bisection method for distance approximation}
\label{alg_dist}
\end{algorithm}

\section{The two-level method for the membership- and cardinality-con\-strained minimum cut problems}
\label{sec:cardinality-membership}

\subsection{Functional for the membership-constrained minimum cut problem}
Our approach to the membership problem is the same two-level procedure as for the unconstrained minimum cut problem, except that the functional \eqref{F-eps} is replaced by the following functional:
For $\eps>0$ and a matrix $E$ of unit Frobenius norm, let $x=(x_i)\in\R^n$ be the eigenvector to the second smallest eigenvalue $\lambda_2$ of $\Lap(W+\eps E)$. Let $\V^-$ and $\V^+$ be the set of indices whose membership to  different components of the cut graph is prescribed. Let $x^-=(x_i^-)$ with $x_i^-=\min(x_i,0)$ 
and $x^+ =(x_i^+)$ with $x_i^+=\max(x_i,0)$ collect the negative and positive components of $x$, respectively. Let $n^-$ and $n^+$ be the numbers of negative and nonnegative components of $x$, respectively. We denote the averages of $x^-$ and $x^+$ by
$$
\langle x^- \rangle = \frac1{n^-} \sum_{i=1}^n x_i^-, \quad\
\langle x^+ \rangle = \frac1{n^+} \sum_{i=1}^n x_i^+.
$$
Motivated by the special form of the eigenvectors as given in the Fiedler theorem (Theorem~\ref{thm:fiedler}), we consider the functional
\begin{equation} \label{F-card}
F_\eps(E) = \lambda_2(\Lap(W+\eps E)) + \frac\alpha 2 \sum_{i\in \V^-} (x_i - \langle x^- \rangle )^2 +
\frac\alpha 2 \sum_{i\in \V^+} (x_i - \langle x^+ \rangle )^2,
\end{equation}
where $\alpha>0$ is a weight to be chosen. The choice of the sign of the eigenvector $x$ is such that $F_\eps(E)$ takes the smaller of the two possible values. This functional is to be minimized under the inequality constraints $W+\eps E\ge 0$, the norm constraint $\|E\|=1$ and the symmetry and the sparsity pattern of $E$.

\subsection{Functional for the cardinality-constrained minimum cut problem}
For the cardinality-constrained problem we use the same functional $F_\eps$, except that the sets $\V^-$ and $\V^+$ are not given {\it a priori}, but are chosen depending on $E$ in the following way: $\V^-$ and $\V^+$ collect the indices of the smallest and largest $\overline n$ components of the eigenvector $x$, respectively, augmented by those indices for which the components of $x$ do not differ by more than a threshold $\delta$ from the average of the smallest and largest $\overline n$ components, respectively.

\subsection{Constrained gradient flow for the functional $F_\eps$}

\subsubsection{Eigenvector derivatives}
We use the following lemma.
\begin{lemma} \label{lem:eigvecderiv} {\rm \cite[Corollary 4]{MeyS88}}
Consider the differentiable $n\times n$ symmetric-matrix valued function $C(t)$ for $t$ in a neighbourhood of $0$,
let $\lambda(t)$ be a simple eigenvalue of $C(t)$ and let $x(t)$
 be the associated eigenvector 
normalized such that $\|x(t)\|_2=1$.
Moreover, let $M(t) = C(t) - \lambda (t) I$ and let $M(t)^\dagger$ be the Moore-Penrose pseudoinverse of $M(t)$. Then,
the derivative of the eigenvector is given by
\begin{equation}
\dot x(t) =   - M(t)^\dagger \dot M(t) x(t).
\label{eq:deigvec}
\end{equation}
\end{lemma}
We remark that in \cite{MeyS88} this is formulated with the group inverse, which in the symmetric case is the same as the Moore-Penrose pseudoinverse.

\subsubsection{Gradient of $F_\eps$} Consider  a differentiable path $E(t)$ of $\eps$-feasible matrices, and denote the corresponding Laplacian matrix by $L(t)= \Lap(W+\eps E(t))$, by $\lambda_2(t)$ the second smallest eigenvalue of $L(t)$ , and by $x(t)$ the associated eigenvector.
We set
$$
\one^-=(\one^-_i)\in \R^n \quad\hbox{ with }\quad \one^-_i =
\begin{cases} 1 \ \hbox{ if } x_i<0 \\ 0 \ \hbox{ else}, \end{cases}
$$
$$
\one^+=(\one^+_i)\in \R^n \quad\hbox{ with }\quad \one^+_i =
\begin{cases} 1 \ \hbox{ if } x_i\ge 0 \\ 0 \ \hbox{ else}, \end{cases}
$$
and, with $e_i$ denoting the $i$th standard unit vector,
$$
v=v^+ + v^- \quad\text{with}\quad v^\pm = - \sum_{i\in\V^\pm} (x_i-\langle x^\pm \rangle) (e_i - \frac1{n^\pm} \one^\pm).
$$
We define 
$$
z = (L-\lambda_2 I)^\dagger v,
$$
which is computed as the solution of the linear system
\begin{equation}\label{z-ls}
\begin{pmatrix}
L-\lambda_2 I & x \\
x^T & 0
\end{pmatrix}
\begin{pmatrix}
z \\
\mu 
\end{pmatrix}
=
\begin{pmatrix}
v \\
0 
\end{pmatrix}.
\end{equation}
We denote by $x\bullet y=(x_iy_i)$ the vector obtained by componentwise multiplication of the entries of $x$ and $y$.
We then have the following result.

\begin{lemma}\label{lem:F-dot-member} In the above situation we have
\begin{align}\label{G-eps3}
	\frac{d}{dt}F_\eps(E) &= \eps \langle  G_\eps(E) , \dot{E} \rangle, \quad\hbox{ where }\\
	G_\eps(E)&=
	P_{\mathcal{E}}\Bigl(\sym\bigl((x\bullet(x +\alpha z)\one ^T  -x (x+\alpha z)^T\bigr)\Bigr)
\end{align}
is symmetric and has the sparsity pattern determined by the set of edges $\mathcal{E}$. 
\end{lemma}

\begin{proof} We have
$$
\frac{d}{dt} \, \frac12 \sum_{i\in \V^-} (x_i - \langle x^- \rangle )^2 = 
\sum_{i\in \V^-} (x_i - \langle x^- \rangle )(\dot x_i - \frac{d}{dt}\langle x^- \rangle )
$$
and similarly for the sum over $\V^+$.
With $K= (L-\lambda_2 I)^\dagger$ we obtain from Lemma~\ref{lem:eigvecderiv} that
$$
\dot x_i = -e_i^T K\dot L x, \qquad \frac{d}{dt}\langle x^\pm \rangle =\frac1{n^\pm}\one^{\pm T} K\dot L x,
$$
so that
$$
\frac{d}{dt} \biggl( \frac12 \sum_{i\in \V^-} (x_i - \langle x^- \rangle )^2 +   \frac12 \sum_{i\in \V^+} (x_i - \langle x^+ \rangle )^2
\biggr)= v^{T}K\dot L x = 
\langle Kv x^T, \dot L \rangle = \langle z x^T, \dot L \rangle.
$$
Using the expression for $\dot L$ given in \eqref{milestone2} and proceeding as in the proof of Lemma~\ref{lem:a-dot} gives the result.
\end{proof}

The computational cost of computing $G_\eps(E)$ lies in computing the second eigenvalue and its eigenvector and in solving the linear system \eqref{z-ls}. For a sparse weight matrix, these computations have a complexity that is linear in the number of vertices.
With this gradient $G_\eps(E)$, the further procedure is now exactly the same as in Section~\ref{sec:two-level method}.

\section{The two-level method for the ambiguity problem}
\label{sec:ambiguity}

\subsection{Two-level formulation}
For the ambiguity problem we proceed similarly as in Section~\ref{sec:two-level method}.
\begin{enumerate}
\item Given $\eps>0$, we look for a symmetric matrix $E\in\R^{n\times n}$ with the same sparsity pattern as $W$ (i.e., $e_{ij}= 0$ if $w_{ij}=0$), of unit Frobenius norm,  with $W+\eps E\ge 0$ (with componentwise inequality)  such that the difference between the third and second smallest eigenvalues of $\mathrm{Lap}(W+\eps E)$ is minimized. The obtained minimizer is denoted by $E(\eps)$.
\item We look for the smallest value of $\eps$ such that the second and third  eigenvalues of 
$\mathrm{Lap}(W+\eps E(\eps))$ coalesce.
\end{enumerate}

In order to compute $E(\eps)$ for a given $\eps>0$, we make use of a constrained gradient system for the functional
\begin{equation}\label{F-eps2}
F_\eps(E) = \lambda_3\bigl( \mathrm{Lap}(W+\eps E) \bigr)-\lambda_2\bigl( \mathrm{Lap}(W+\eps E) \bigr),
\end{equation}
under the inequality constraints $W+\eps E\ge 0$, the norm constraint $\|E\|=1$ and the symmetry and the sparsity pattern of $E$.

In the outer iteration we compute the optimal $\eps$, denoted $\eps^\star$, by a combined Newton-bisection method as in Section~\ref{sec:two-level method}.

\subsection{Gradient of $F_\eps$}
Consider  a regular path $E(t)$ of $\eps$-feasible matrices, and denote the corresponding Laplacian matrix by $L(t)= \Lap(W+\eps E(t))$ and by $\lambda_2(t)$ and $\lambda_3 (t)$ the second and third smallest eigenvalues of $L(t)$, respectively.
We denote by  $x(t)$ and $y(t)$ the  corresponding eigenvectors of unit Euclidean norm.

\begin{lemma}\label{lem:a-dot-amb} In the above situation we have
\begin{align*}\label{G-eps4}
	\frac{d}{dt}F_\eps(E) &= \eps \langle  G_\eps(E) , \dot{E} \rangle, \quad\hbox{ where }\\
	G_\eps(E)&=- P_{\mathcal{E}}\bigl(\sym(x^2 \one ^T)-x x^T-\sym(y^2 \one ^T)+yy^T\bigr)
\end{align*}
is symmetric and has the sparsity pattern determined by the set of edges $\mathcal{E}$. 
\end{lemma}
 
With this gradient we then proceed further as in Section~\ref{sec:two-level method}.

\section{Algorithmic aspects}
\label{sec:alg-aspects}

\subsection{Discretizing the constrained gradient flow} \label{subsec:euler}
We use a modified explicit Euler method for the approximate integration of the differential equation \eqref{cgf}. For a given $\eps>0$, a step-size $h>0$ and from the $\eps$-feasible perturbation matrix $E^n$ of the $n$th time step, we compute $E^{n+1}$ as follows.  We compute $G_\eps(E^n)=\bigl(g_{ij}^n\bigr)$ and define $\widetilde E^{n+1}=\bigl(\tilde e_{ij}^{n+1}\bigr)$ by setting
$$
\tilde e_{ij}^{n+1} = e_{ij}^n - h g_{ij}^n \qquad \text{if}  \quad \ w_{ij}+ \eps\bigl( e_{ij}^n - h g_{ij}^n \bigr) \ge 0
$$
and else 
$$
\tilde e_{ij}^{n+1} = e_{ij}^n - \theta h g_{ij}^n \quad\ \text{ with $\theta\in[0,1)$ such that } 
w_{ij}+ \eps\bigl( e_{ij}^n - \theta h g_{ij}^n \bigr) = 0,
$$
that is, with
$
\theta = ({w_{ij}+ \eps e_{ij}^n})/({\eps h g_{ij}^n}).
$
We would ideally take the new perturbation matrix $E^{n+1}$ such that 
$$
\| E^{n+1} - \widetilde E^{n+1} \| \to \min  \quad\text{ subject to}\quad
\| E^{n+1} \| =1 \ \text{ and } \ W+\eps E^{n+1}\ge 0.
$$
We approximate this optimization problem by treating the two constraints one after the other in an alternating way. 
Let $\E_0$ be the set of edges $(i,j)\in\E$ for which $w_{ij}+\eps \tilde e_{ij}^{n+1} = 0$ (cut edges), and let 
$P^0 = P_{\E_0}$ and $P^+=P_{\E\setminus \E_0}$ be the complementary projections as defined in Section~\ref{subsec:gradient}.
We first normalize by choosing
$\rho>0$ such that
$$
\widehat E^{n+1} = P^0 \widetilde E^{n+1} + \rho P^+ \widetilde E^{n+1}
$$
has unit Frobenius norm, i.e.,
$$
\rho= \frac{\sqrt{1-\| P^0  \widetilde E^{n+1} \|^2}}{ \| P^+ \widetilde E^{n+1} \|^2}.
$$
(In case that $\| P^0  \widetilde E^{n+1} \|$ is larger than 1 or very close to 1, we replace $P^0  \widetilde E^{n+1}$ by $P^0E^n$ in the two lines above.)
We denote by $\E_-$ the set of inadmissible edges $(i,j)$ for which $w_{ij}+\eps \hat e_{ij}^{n+1} < 0$. We then  reset $\hat e_{ij}^{n+1}$ to 
$$
e_{ij}^{n+1} = -\frac{w_{ij}}\eps \quad\ \text{ for }\quad (i,j)\in \E_-,
$$
augment $\E_0 := \E_0\cup \E_-$ and consider the updated projection $P^0=P_{\E_0}$. We then normalize the so obtained matrix $E^{n+1}$ in the same way as above by leaving the entries of $P^0E^{n+1}$ unchanged, reset the entries for inadmissible edges, normalize, and so on. As there are only finitely many edges, this iteration terminates after finitely many steps (typically after the first step). Finally, we have obtained an $\eps$-feasible perturbation matrix $E^{n+1}$.

\subsection{Choice of step-size}
The step-size $h$ can, for example, be selected by the following adaptive algorithm. Here the objective is to reduce the function $F_\eps$, not to follow accurately a trajectory of the constrained gradient differential equation.
%
%
			%
			%

\IncMargin{1em}
\begin{algorithm}
\DontPrintSemicolon
\KwData{Matrix $E_{n}$ and stepsize $h_{n-1}$ are given}
\KwResult{Matrix $E_{n+1}$ and stepsize $h_n$}
\Begin{
\nl Initialize the step-size by the previous step-size, $h=h_{n-1}$\;
\nl Compute $E_{n+1}(h)$ and its function value $F_\eps(E_{n+1}(h))$ with the step-size $h$\;
\eIf{$F_\eps(E_{n+1}(h)) \ge F_\eps(E_{n})$}
{
halve the step-size, $h:=h/2$ and repeat from 2
}
{
\If{$h=h_{n-1}$}{compute $E_{n+1}(2h)$ and its function value $F_\eps(E_{n+1}(2h))$ with the step-size $2h$}
\If{$F_\eps(E_{n+1}(2h)) \le F_\eps(E_{n+1}(h))$}{double the step-size, $h:=2h$}
Set $h_n=h$ and $E_{n+1}=E_{n+1}(h)$\;
}
\nl Return $E_{n+1}$ and $h_n$\;
}
\caption{Step-size selection}
\label{alg_step}
\end{algorithm}

\subsection{Stopping criterion} 
Let $F^n=F_\eps(E^n)$. In order to stop the integration when $F^n$ has approximately reached a stationary value, we use a criterion of the following type:
$$
\text{Integrate until }\quad F^n - F^{n+1} \le \beta h F^n+\delta \quad\text{or} \quad F^n \le \text{tol},
$$
where tol is a tolerance parameter (e.g., tol = $10^{-6}$) and $\beta$ and $\delta$ are further parameters. We had good experience with the choice $\beta=10\,\cdot\,$tol and $\delta=\text{tol}/100$; see further \cite{edelmann-master} where also the sensitivity of the algorithm to the chosen parameters is discussed.

\subsection{Initial value of the constrained gradient flow for a new $\eps$} When we change to a new value of $\eps$ in the outer iteration, we need an initial value for the constrained gradient flow. A first idea might be to take the terminal perturbation matrix $\widetilde E^0=E(\eps_{\rm old})$ as the initial value, but usually this does not satisfy the nonnegativity constraints $W+\eps E^0\ge 0$ if $\eps\ge\eps_{\rm old}$. We therefore modify $\widetilde E^0$ to $E^0$ by solving approximately
$$
\| E^{0} - \widetilde E^{0} \| \to \min  \quad\text{subject to}\quad
\| E^{0} \| =1 \ \text{ and } \ W+\eps E^0\ge 0,
$$
alternating between normalization and enforcing the nonnegativity constraints as in Section~\ref{subsec:euler}, but this time beginning with the empty set $\E_0=\emptyset$.

\subsection{Choice of the inital perturbation size $\eps_0$ and the initial perturbation matrix} 
While one might just start with a random perturbation, a more educated guess starts from the normalized free gradient  $E^0=-G_\eps(0)/\| G_\eps(0)\|$ and determines $\eps_0$ as the largest number $\eps$ such that $W+\eps E^0\ge 0$.

\subsection{Stopping the outer iteration before convergence} In the exact solution $W^\star=W+\eps^\star E^\star$ to the constrained minimum cut problem, the entries of $W^\star$ are either zero or those of $W$. To decide about the cut, it is therefore not necessary to iterate towards $W^\star$ with very high accuracy, but instead the cut can be inferred earlier from a moderately accurate approximation $W+\eps E(\eps)$, for example using the following criterion, with a small threshold parameter $\vartheta>0$:
$$
\text{Stop if for every edge $(i,j)\in\E$, either  $w_{ij} +\eps e_{ij} \le \vartheta w_{ij}$ or $|\eps e_{ij}| \le \vartheta w_{ij}$.}
$$
In the first case one would then cut to $w_{ij}^\star=0$ and in the second case one would leave the weight unchanged: $w_{ij}^\star=w_{ij}$. It can finally be checked if the so obtained cut graph is indeed disconnected, by computing $\lambda_2(\Lap(W^\star))$. Instead, in the ambiguity problem such a shortcut is not feasible.



%
\section{Numerical examples}
\label{sec:num-examples}

We consider a few illustrative examples for both the cardinality and the membership constraints. At the end we will also consider graphs 
to illustrate the ambiguity problem.


The standard Fiedler spectral partitioning algorithm, to which we refer below, is simply based on the sign of the components 
of the eigenvector of $\Lap(W)$ associated to $\lambda_2$.

\begin{ex}[Zachary's karate club]\label{ex:zac}
\rm 
This  weighted graph consisting of $34$ vertices describes 
the relationship between $34$ members of a karate club (for a detailed 
description see \cite{Zac77}).
Using the Fiedler spectral partitioning we obtain two connected components 
 of $16$ and $18$ vertices as shown in Figure \ref{fig:zachari1}.

According tho the standard Fiedler partitioning the first component is 
led by the vertex labeled as $1$, while the second one is led 
by the vertex labeled as $34$. 

We next consider the following constraints:
\begin{enumerate}
\item[(i)] Cardinality constraint with threshold equal to $\bar n=17$ vertices;
\item[(ii)] Membership constraint;
\item[(iii)] Both constraints.
\end{enumerate}
In more detail:
\begin{enumerate}
\item[(i)] By asking for a cardinality constraint with $\bar n=17$ vertices in each component, the approximate computed 
distance is $\eps^\star\approx 10.03631$. 
The results in Table \ref{tab:ex} are obtained by setting a tolerance ${\rm tol}=10^{-5}$ 
and the weight $\alpha=3$ in \eqref{F-card}.
With respect to the standard partition obtained by the Fiedler eigenvector of $\Lap(W)$, the vertex that 
changes partition is the vertex labeled as $9$.

\begin{table}[h!]
  \centering
 \caption{Computed values of $\varepsilon$, $f(\varepsilon)=F_\eps(E(\eps))$  for Example \ref{ex:zac},\textrm{(i)}}
  \label{tab:ex}
  \begin{tabular}{|c|l|l|}
\hline
    k & $\varepsilon_k$  & $f(\varepsilon_k)$  \\
    \hline\hline
0&   1.355198757424337 &  5.000000000000000\\
1&   1.142857142857148 & 10.036313891705188\\
2&   0.000001319577846 & 10.036313891705202\\

\hline
  \end{tabular}
\end{table}

\item[(ii)] 
In this second case we consider the membership constraint: we ask for the vertices
$1$ and $34$ to be in different partitions. Moreover we consider $4$ different 
cases, that is vertex $9$ to be in the same connected component as vertex $1$, vertex $32$ 
to be in the same component as vertex $1$, vertex $14$ to be in the same component 
as vertex $34$, or vertex $20$ in the same component  as vertex $34$.

For these four examples,  Table \ref{tab:ex2} reports the values of 
$\varepsilon^\star$ 
for the functional $F_\eps$ of  \eqref{F-card} with $\alpha=3$,
as computed
by setting the tolerance ${\rm tol}=10^{-5}$. 
We can see from Table \ref{tab:ex2} that vertex $9$ is the easiest to be required for 
changing the connected component; on the other hand, vertex $20$ turns out to be the most difficult 
to change the component.

\begin{figure}[h!]
\centering
\subfloat[][Zachary's karate club colored by standard Fiedler partitioning.]{\label{fig:zachari1}\includegraphics[width=.46\textwidth]{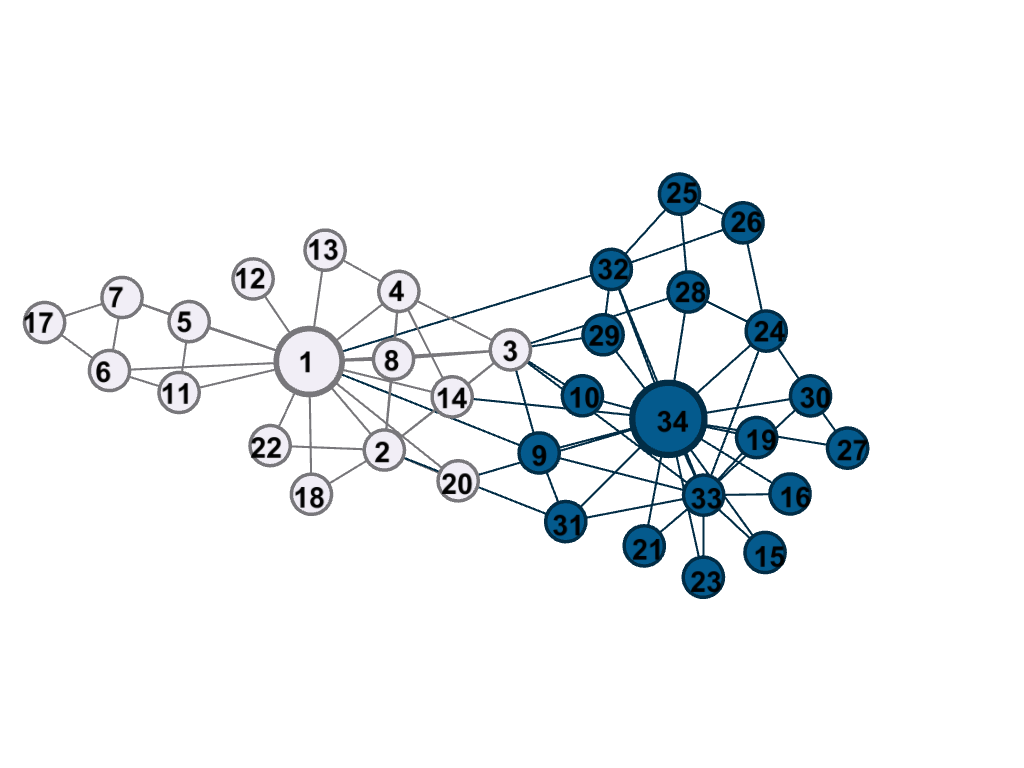}}\quad
\subfloat[][Zachary's karate club-membership constraint: highlighted the vertices required to change partition with respect to the clustering shown in (a).]{\label{fig:zachari2}\includegraphics[width=.46\textwidth]{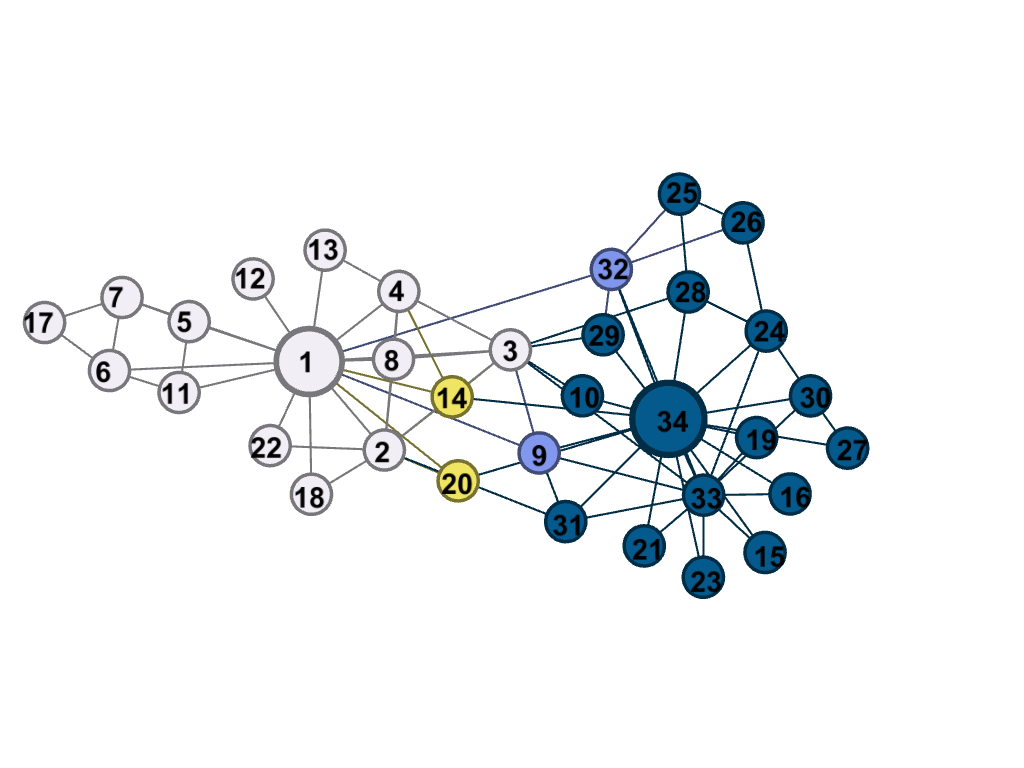}}
\caption{Example \ref{ex:zac}: Zachary's karate club}
\label{fig:fig}
\end{figure}

\begin{table}[h!]
  \centering
 \caption{Computed values of $\varepsilon^{\star}$ for Example \ref{ex:zac},\textrm{(ii)}}
  \label{tab:ex2}
  \begin{tabular}{|c|l|}
\hline
    node & $\varepsilon^{\star}$   \\
    \hline\hline
   9 &  16.947756820436005 \\
  14 & 19.816423934360159  \\
  20 & 26.394452875575567  \\
  32 & 19.849724386431539  \\
\hline
  \end{tabular}
\end{table}

\item[(iii)] Finally we consider both constraints, that is, we ask for a cardinality constraint with threshold $\bar n=17$ vertices, and we require that vertex $9$ is in the same connected component as the vertex $34$.
Table \ref{tab:ex3} shows values of $\varepsilon_k$, $f(\varepsilon_k)=F_{\eps_k}(E(\eps_k))$ computed with the tolerance ${\rm tol}=10^{-5}$ and the weights $\alpha_c=3,\ \alpha_m=10$ in the functional $F_\eps$ that combines the cardinality and membership functionals.
We obtain that the vertex $10$ further changes the connected component in order to satisfy the cardinality constraint.
\end{enumerate}

\begin{table}[h!]
  \centering
 \caption{Computed values of $\varepsilon_{k}$, $f(\varepsilon_{k})$ for Example \ref{ex:zac},\textrm{(iii)}}
  \label{tab:ex3}
  \begin{tabular}{|c|l|l|}
\hline
    k & $\varepsilon_{k}$   & $f(\varepsilon_{k})$ \\
    \hline\hline
0 &  1.401034325554263 &  5.000000000000000 \\
1 &  1.279411764675930 & 10.206652233031356 \\
2 &  0.000000000000002 & 10.206652233031399 \\

\hline
  \end{tabular}
\end{table}
\end{ex}

\begin{ex}[A misbehavior of the algorithm]\label{ex:failure} 
\rm
We present an example where the algorithm fails.

Consider an unweighted graph with $N$ vertices, such that each vertex $2, \ldots, N$ is connected to the following two vertices, i.~e.
\[ w_{i,i+1} = w_{i,i+2} = 1 \text{ for all }2 \le i \le N-2 \,.\]
The first vertex is connected to the second one, i.e., $w_{1,2} = 1$, but not to the third one. 
\begin{figure}[h!]
\centering{
\begin{tikzpicture}[scale=0.9]
	\Vertex[x=0,y=0]{1}
    \Vertex[x=2,y=0]{2}
    \Vertex[x=4,y=0]{3}
    \Vertex[x=6,y=0]{4}
    \Vertex[x=8,y=0]{5}
    \Vertex[x=10,y=0]{6}
    \Vertex[x=12,y=0]{7}
    \Vertex[x=14,y=0]{8}
    \Edge(1)(2)
    \Edge(2)(3)
    \Edge(3)(4)
    \Edge(4)(5)
    \Edge(5)(6)
    \Edge(6)(7)
    \Edge(7)(8)
    \tikzstyle{EdgeStyle}=[bend left]
    \Edge(2)(4)
    \Edge(4)(6)
    \Edge(6)(8)
    \tikzstyle{EdgeStyle}=[bend right]
    \Edge(3)(5)
    \Edge(5)(7)
\end{tikzpicture}
}
\caption{Graph with 8 vertices.}
\label{ex:counter}
\end{figure}
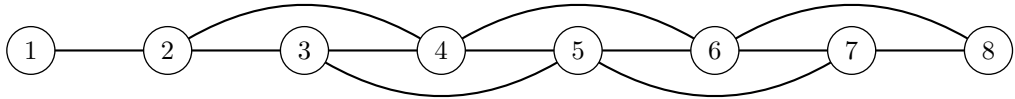
It is clear that the minimum cut is obtained by removing the edge $(1,2)$, since this is the only possibility to obtain a disconnected graph when only one edge is removed. 
We tried to solve this example for different values of $N$.
The algorithm works correctly when $N=8$ but fails when $N \ge 12$. 
In the latter case we obtain a disconnected graph, but with the wrong edges removed. 
For $N=20$ 
the resulting partition is $\{1,2,\ldots,10\} \cup \{11,12,\ldots,20\}$ 
instead of the correct partition $\{1\}\cup\{2,3,\ldots,20\}$.
If we impose a cardinality constraint with $\bar n=10$ we get the same solution.
\end{ex}

\begin{figure}[h!]
\centering
\subfloat[][Books about US Politics colored by Fiedler eigenvector.]{\label{fig:pol_fiedler2}
\includegraphics[width=.45\textwidth]{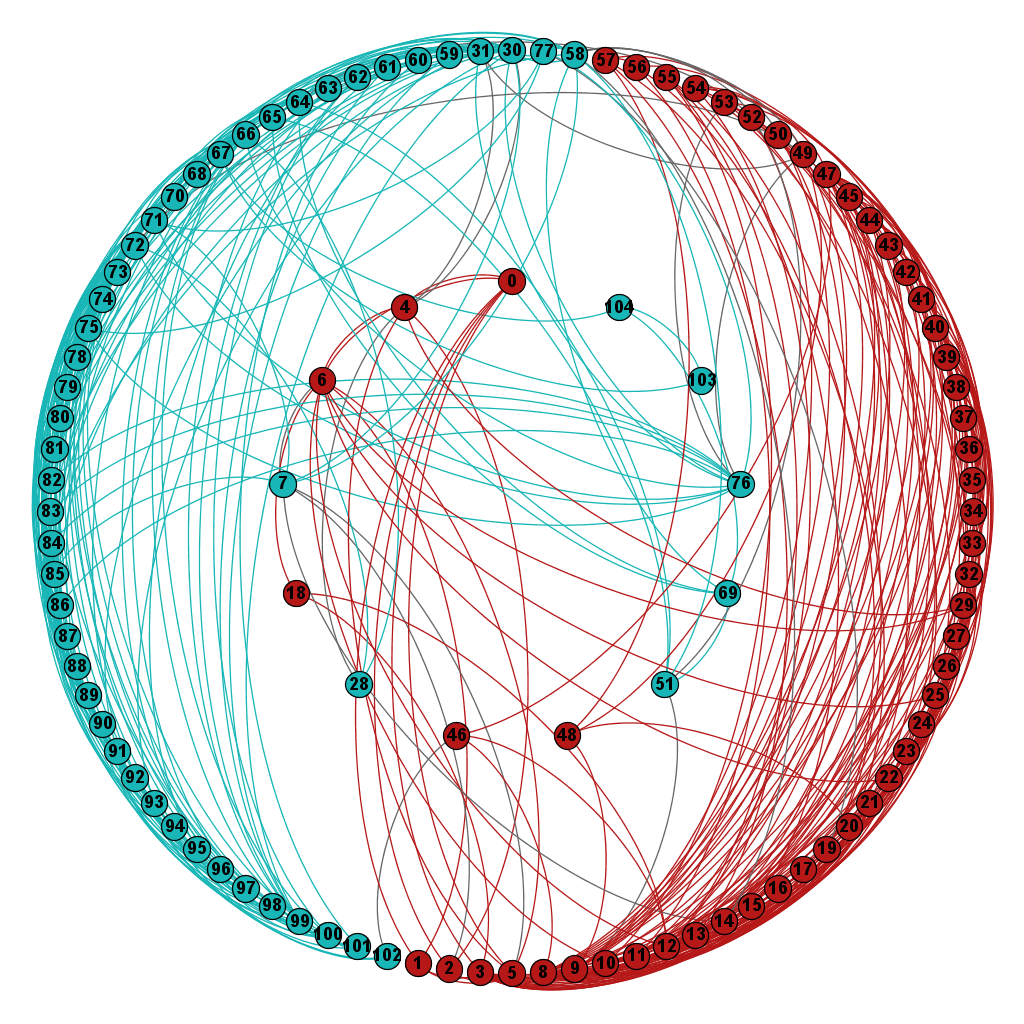}} \qquad
\subfloat[][Books about US Politics colored red, blue or green, based on book buying data.]
{\label{fig:pol_orig2}\includegraphics[width=.45\textwidth]{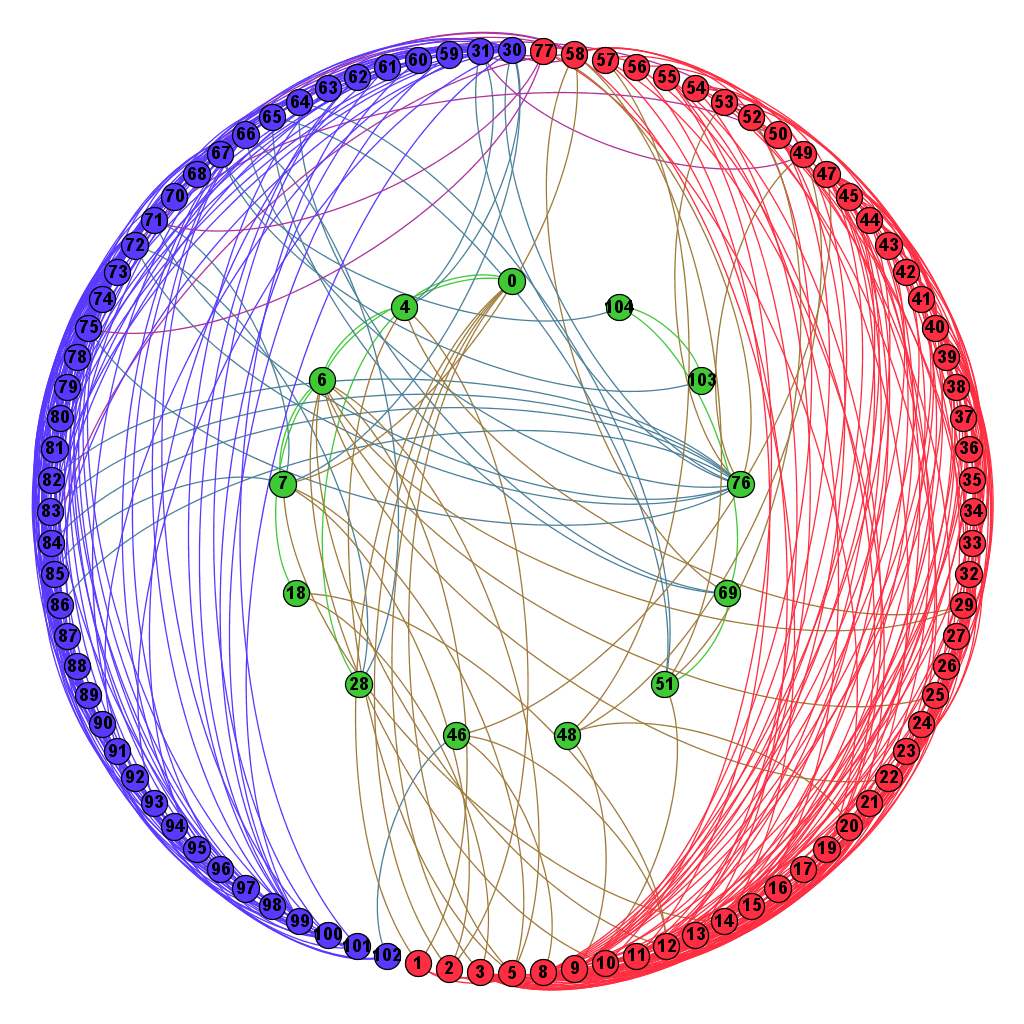}}

\caption{Example \ref{ex:pol}: Books about US Politics}
\label{fig:fig2}
\end{figure}

\begin{ex}[Books about US Politics]\label{ex:pol}
\rm This graph is a network of $105$ vertices, each one representing a book about US politics sold in 2004 by an online bookseller \cite{krebs04}. 
Two books are linked if they were purchased by the same person, and the books are colored red, blue or green, based on book buying data (see Fig.(\ref{fig:pol_orig2})). 
The links determine the grouping and coloring of the vertices. 
The vertices are colored by Fiedler spectral partitioning in Figure \ref{fig:pol_fiedler2}. By this partitioning we obtain two connected components of $52$ and $53$ vertices. 
We ask for the membership constraints: we compute the distance (see Table \ref{tab:expol}) for each one of the vertex in the green group of Figure \ref{fig:pol_orig2} to belong to the blue group ($\varepsilon_b^{\star}$) or to belong to the red group ($\varepsilon_r^{\star}$).
The results reported in the tables are obtained with 
the weight $\alpha=1$. 

\begin{table}[h!]
  \centering
 \caption{Computed values of $\varepsilon^{\star}$ for Example \ref{ex:pol}}
  \label{tab:expol}
  \begin{tabular}{|c|l|l|}
\hline
    vertex & $\varepsilon_b^{\star}$   & $\varepsilon_r^{\star}$ \\
     \hline\hline
0&  19.550643238475178  &  10.139259626395429 \\
4&  17.599307526562065  &  10.139259626395429 \\
6&  80.390748501362353  &  10.139259626395429 \\
7&  10.277273821728212  &  14.165678052077872 \\
18& 31.581044248528972  &  10.139259626395429 \\
28&  10.139259626395429 &  38.150041659347089 \\
46&  12.674010726779485 &  10.139259626395429 \\
48&  72.656404890260845 &  22.909386776883416 \\
51&  13.625085422874610 &  17.008619024532120 \\
69&  10.139259626395429 &  26.824720735807777 \\
76&  22.909386776883416 &  30.612769763645328 \\
103& 10.139259626395429 &  13.668979290979349 \\
104& 10.139259626395429 &  14.953186306634851 \\
\hline
\end{tabular}
\end{table}

\end{ex}

\begin{ex}[Les Miserables]\label{ex:lesmis}
\rm Figure \ref{fig:fig3} shows the graph of character co-occurence in Les Miserables \cite{Knuth:2009:SGP:1610343}.  
This graph consists of $77$ vertices (representing characters).
According to the Fiedler partitioning, $22$ of these belong to one part and the remaining $55$ belong to the other part.
Asking for the partitioning of the graph with the cardinality constraint with threshold $\bar n=35$, we obtain the
result shown in Figure \ref{fig:fig3}. 


\begin{figure}[h!]
\centering
\includegraphics[width=.55\textwidth]{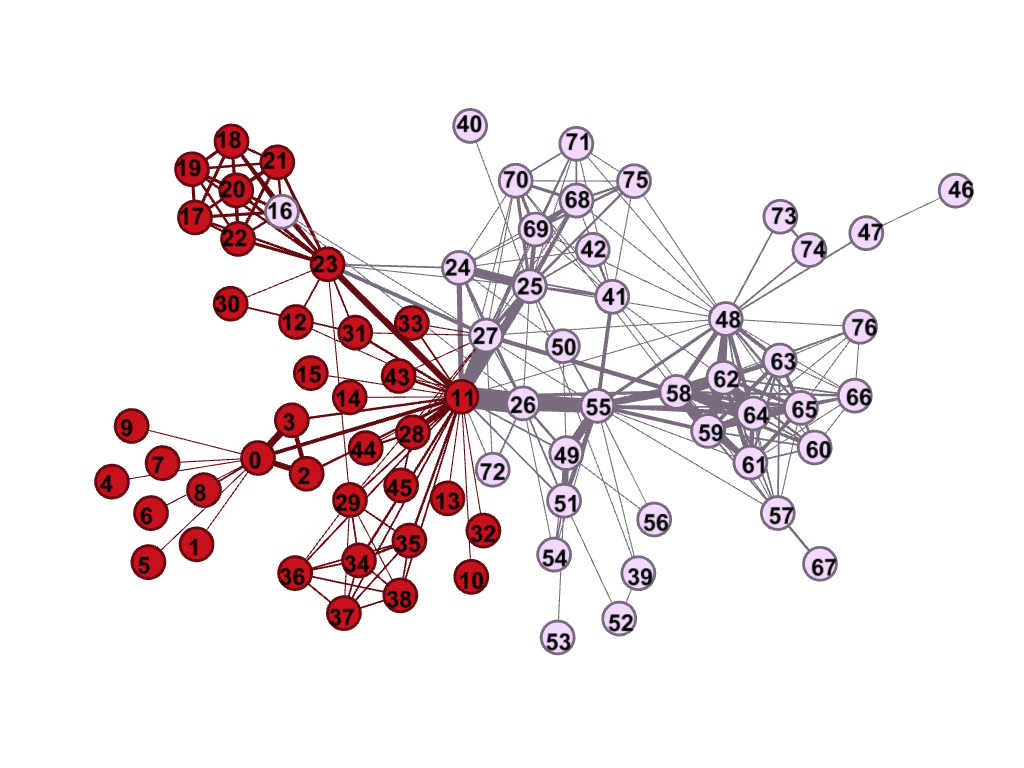}
\caption{Example \ref{ex:lesmis}: Les Miserables, cardinality-constrained graph partitioning}
\label{fig:fig3}
\end{figure}
\end{ex}


\begin{ex}[Planted Partition Model --- ambiguity problem]\label{ex:plant}
\rm We consider here a class of graphs for which we investigate the distance to ambiguity (that is, a coalescence of the second and third eigenvalues in the 
associated Laplacian matrix).  We consider 
Planted Partition Models \cite{Condon:2001:AGP:373515.373517}, which constitute a special case of Stochastic Block models, a commonly used generative model for social and biological networks. The probability matrix consists of a constant value $p_{in}$ on the diagonal and a different constant value $p_{out}$ off the diagonal; in addition, the number of vertices is $n=40$, while the communities are $4$, each one made of $10$ vertices.
In Figure \ref{fig:fig2} we can see two Planted Partition Models; on the left when $p_{in}=0.8$ and $p_{out}=0.2$, on the right when $p_{in}=0.9$ and $p_{out}=0.1$.
In Table \ref{tab:ex6} a comparison among various values of $p_{in}$ and $p_{out}$ is shown.
These results are obtained by setting the tolerance to $tol=10^{-5}$.
\end{ex}

\begin{figure}[h!]
\centering
\subfloat[][Sparsity pattern for PPM with $p_{in}=.8$ and $p_{out}=.2$]{\label{fig:plant1}\includegraphics[width=.40\textwidth]{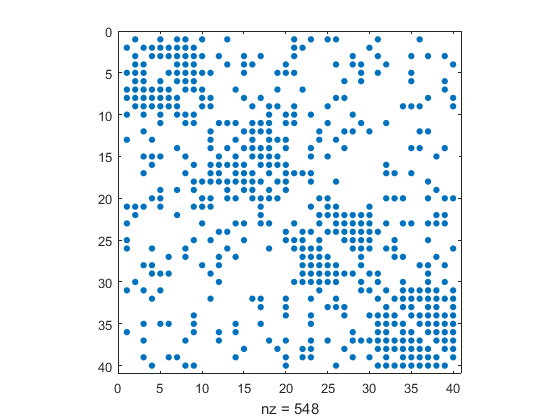}}\quad
\subfloat[][Sparsity pattern for PPM with $p_{in}=.9$ and $p_{out}=.1$]{\label{fig:plant3}\includegraphics[width=.40\textwidth]{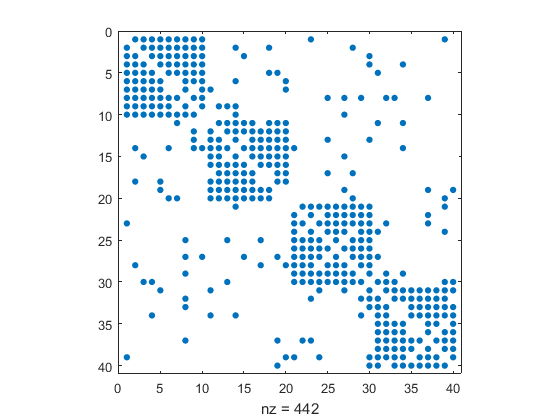}}
\caption{Example 4: Planted Partition Model}
\label{fig:fig4}
\end{figure}

%

\begin{table}[h!]
  \centering
 \caption{Example \ref{ex:plant}: Computed values of $\varepsilon^{\star}$ for  various parameters}
  \label{tab:ex6}
  \begin{tabular}{|c|l|l|l|l|}
\hline
    $p_{in}$ & $p_{out}$ & $\varepsilon^{\star}$   \\
    \hline\hline
   0.80 & 0.20 & 1.310680592143721   \\
   0.85 & 0.15 & 1.011621669775467   \\
   0.90 & 0.10 & 1.068267456259814   \\
   0.95 & 0.05 & 0.848607315993027   \\
   1.00 & 0.00 & 0  \\
\hline
  \end{tabular}
\end{table}

The results show that relatively small perturbations may yield an ambiguity in the Fiedler partitioning for these graphs.
 
\section*{Acknowledgments}

The authors thank Daniel Kressner (EPFL, Lausanne) for interesting discussions during an Oberwolfach meeting and
Armando Bazzani (University of Bologna, Italy) for stimulating discussions.

Part of this work was developed during some visits to Gran Sasso Science Institute in L'Aquila and to the University of T\"ubingen.
The authors thank both institutions for the very kind hospitality.

N. Guglielmi thanks the Italian M.I.U.R. and the INdAM GNCS for financial support and also the Center of Excellence DEWS.


\end{document}